\documentclass[oneside,english]{amsart}
\usepackage[T1]{fontenc}
\usepackage[latin9]{inputenc}
\usepackage{geometry}
\usepackage{amstext}
\usepackage{amsthm}
\usepackage{amssymb}
\usepackage{graphicx}
\usepackage{hyperref}
\usepackage{bbm}

\makeatletter
\numberwithin{equation}{section}
\numberwithin{figure}{section}
\theoremstyle{plain}
\newtheorem{thm}{\protect\theoremname}
  \theoremstyle{remark}
  \newtheorem{rem}[thm]{\protect\remarkname}
  \theoremstyle{definition}
  
  \theoremstyle{plain}
  
  \newtheorem{lem}[thm]{\protect\lemmaname}
  \newtheorem{defn}[thm]{\protect\defname}
\makeatother

\newcommand{\R}{\mathbb{R}}
\newcommand{\E}{\mathbb{E}}
\newcommand{\B}{\mathcal{B}}
\newcommand{\C}{\mathcal{C}}
\newcommand{\D}{\mathcal{D}}
\newcommand{\Bf}{\mathfrak{B}}
\newcommand{\U}{\mathcal{U}}
\newcommand{\V}{\mathcal{V}}

\DeclareMathOperator{\Li}{Li}
\DeclareMathOperator{\sech}{sech}

\usepackage{babel}
  \providecommand{\corollaryname}{Corollary}
  \providecommand{\examplename}{Example}
  \providecommand{\remarkname}{Remark}
  \providecommand{\defname}{Definition}
  \providecommand{\lemmaname}{Lemma}
\providecommand{\theoremname}{Theorem}

\begin{document}

\title{A Symbolic Approach to the Poly-Bernoulli Numbers}

\author{Tanay Wakhare$^{\dag}$ and Christophe Vignat$^{\ast}$}
\address{$^\ast$~Tulane University, New Orleans, LA 70118, USA and Universit\'{e} Paris Sud, France}
\email{cvignat@tulane.edu, christophe.vignat@u-psud.fr}
\address{$^{\dag}$~University of Maryland, College Park, MD 20742, USA}
\email{twakhare@gmail.com}

\maketitle

\begin{abstract}
We present a symbolic representation for the poly-Bernoulli numbers. This allows us to prove several new iterated integral representations for the poly-Bernoulli numbers, including an integral transform of the Bernoulli-Barnes numbers. We also deduce some new recurrences for the poly-Bernoulli numbers. Finally, we use these results to present a new iterated integral representation for the Arakawa-Kaneko zeta function, including  a nonlinear integral transform of the Barnes zeta function.
\end{abstract}

\section{Introduction}
In their seminal 1999 work, Arakawa and Kaneko \cite{ArakawaKaneko} introduced their namesake zeta function, defined by 
\begin{equation}\label{akdef}
\zeta_k(s) := \frac{1}{\Gamma(s)} \int_0^{\infty} \frac{t^{s-1}}{e^t-1}\Li_k\left(1-e^{-t}\right) dt
\end{equation}
where $\Li_{k}\left( z \right)$ is the polylogarithm function defined by
\[
\Li_{k}\left( z \right) = \sum_{n=1}^{+\infty} \frac{z^n}{n^k}.
\]
The importance of this function is that when $s=m$ is a positive integer, we have the evaluation
\begin{equation}
\label{zetakm}
\zeta_k(m)  = \zeta^\ast(k+1,\{1\}^{m-1}),
\end{equation}
where 
$$\zeta^\ast({s_1,\ldots,s_k}):= \sum_{n_1\geq n_2>\cdots \geq n_k\ge 1} \frac{1}{n_1^{s_1}n_2^{s_2}\cdots n_k^{s_k}} , s_1\geq 2, s_i\geq 1, 2\leq i \leq k$$
is a \textit{multiple zeta starred value} (MZSV). Therefore, we can study $\zeta_k(s)$ complex analytically, which specializes to nontrivial relations for the discretized MZSVs. MZSVs are natural generalizations of the Riemann zeta function which have been systematically studied since the 1990s \cite{Zagier}. They occur naturally in the calculation of higher order Feynman diagrams and renormalization constants in physics \cite{broadhurst}, and characterizing all linear dependence relations between MZSVs and the closely related multiple zeta values (MZVs)
$$\zeta({s_1,\ldots,s_k}):= \sum_{n_1 > n_2>\cdots > n_k\ge 1} \frac{1}{n_1^{s_1}n_2^{s_2}\cdots n_k^{s_k}} , s_1\geq 2, s_i\geq 1, 2\leq i \leq k,$$ has become a hot 
topic of recent research. For an excellent survey 
of this topic, we recommend the review article of Zagier \cite{Zagier}.

We now consider the \textit{poly-Bernoulli numbers} $B_n^{(k)}$, defined by the generating function
\begin{equation}\label{polydef}
\sum_{n\ge0}\frac{B_{n}^{\left(k\right)}}{n!}t^{n}=\frac{\Li_{k}\left(1-e^{-t}\right)}{1-e^{-t}},
\end{equation}
and their polynomial extension, the \textit{poly-Bernoulli polynomials} with generating function

\begin{equation}\label{polynomialpolydef}
\sum_{n\ge0}\frac{B_{n}^{\left(k\right)}\left( z \right)}{n!}t^{n}=\frac{\Li_{k}\left(1-e^{-t}\right)}{1-e^{-t}}e^{zt}.
\end{equation}

The importance of this definition is that for a positive integer $n,$ we have the explicit value for the analytic continuation $\zeta_{k}(-n) = (-1)^nB_n^{(k)}$, generalizing the classical result $\zeta(-n) = (-1)^n\frac{B_{n+1}}{n+1}$ for the Riemann zeta function \cite{Arakawa}. We wish to study relations for the poly-Bernoulli numbers; these systematically translate to relations for the analytic continuation of the Arakawa-Kaneko zeta function. In Section \ref{sec2}, we provide the necessary bakground on umbral calculus and the Bernoulli symbol. In Section \ref{sec3} we provide our main results: a symbolic representation for the poly-Bernoulli numbers, and a new integral representation for the poly-Bernoulli numbers in terms of Bernoulli-Barnes numbers. This allows us to derive several old and new recurrences for the poly-Bernoulli numbers in Section \ref{sec4}. Finally, in Section \ref{sec5} we provide new integral and symbolic expressions for the Arakawa-Kaneko zeta function.

\section{Umbral Background }\label{sec2}
The results in this paper depend on \textit{umbral calculus}, a symbolic computation method for which we provide a short introduction. The key idea is that we can express Bernoulli numbers in terms of moments: more precisely \cite{Vignat1}, 
\begin{equation}\label{bernint}
B_n = \frac{\pi}{2}\int_{-\infty}^\infty \left( \imath t-\frac{1}{2} \right)^n \sech^2\left(  \pi t \right) dt,\,\,n\ge 0.
\end{equation}
Given a random variable $\mathcal{L}$ distributed according to the secant square law  $\mathcal{L} \sim \frac{\pi}{2}\sech^2(\pi t)$, the Bernoulli numbers are therefore the moments $B_n = \E \left( i \mathcal{L} - \frac12\right)^n$, where $\E$ denotes epectation value. This interpretation extends to polynomials; the Bernoulli polynomials are the expectation $$B_n\left( x \right)=\E\left(i\mathcal{L} +x -\frac12\right)^n =  \frac{\pi}{2}\int_{-\infty}^\infty \left( \imath t-\frac{1}{2} +x \right)^n \sech^2\left(  \pi t \right) dt.$$
Alternatively, the Bernoulli polynomials  are defined by the generating function
\begin{equation}\label{bpolygenfunc}
\sum_{n=0}^\infty B_n(z)\frac{t^n}{n!} = \frac{e^{zt} t}{e^t-1}, 
\end{equation}
and the Bernoulli numbers as $B_n = B_n(0)$.
Therefore, any occurrence of a Bernoulli number may be replaced with an equivalent \textit{Bernoulli symbol} (or \textit{Bernoulli umbra}) $\mathcal{B}$, such  every Bernoulli number ${B}_{n}$ is mapped to a power $\mathcal{B}^n$. We then perform whatever operations we want, then apply the ``evaluation map" $eval(\mathcal{B}^n) = B_n$, a linear functional. This corresponds to replacing any occurrence of $B_n$ with $\left( \imath t-\frac{1}{2} \right)^n$, performing some manipulations, then multiplying by $\sech^2(\pi t)$ and integrating across $\mathbb{R}$. For more details about this approach, we recommend \cite{Gessel}.

We then have the evaluation rules for the Bernoulli symbols, which essentially coincide with those of probabilistic expectation:
\begin{itemize}
\item A product of several Bernoulli numbers is replaced with independent Bernoulli umbr{\ae} according to 
$B_{n_1}B_{n_2} \mapsto \mathcal{B}_{1}^{n_1}\mathcal{B}_{2}^{n_2}$
\item For two identical symbols, $\mathcal{B}_1^{n_1}\mathcal{B}_1^{n_2}=\mathcal{B}_1^{n_1+n_2}$
\item We have the periodicity relation $\mathcal{B}+1 = - \mathcal{B}$.
\end{itemize}
For example, the simple recurrence 

\[
\sum_{k=0}^n \binom{n}{k} B_k = (-1)^n B_n,\,\, n\ge 0
\]
can be derived through umbral means:
$$\sum_{k=0}^n \binom{n}{k} B_k   \underset{\mathrm{umbralize}}{\mapsto} \sum_{k=0}^n \binom{n}{k} \mathcal{B}^k = (\mathcal{B}+1)^n = (-1)^n \mathcal{B}^n \underset{\mathrm{eval}}{\mapsto} (-1)^n B_n.$$
If we explicitly write out the integrals underlying this and note that $\mathcal{L}$ is even, we have performed the operations
\begin{align*}
\sum_{k=0}^n \binom{n}{k} B_k   &\underset{\mathrm{umbralize}}{=} \sum_{k=0}^n \binom{n}{k} \E \left( i \mathcal{L} - \frac12\right)^k \\
&= \E \left[ \sum_{k=0}^n \binom{n}{k}  \left( i \mathcal{L} - \frac12\right)^k \right] \\
&= \E  \left( i \mathcal{L} - \frac12 + 1\right)^n \\
&=(-1)^n \E  \left( i \mathcal{L} - \frac12 \right)^n \\
& \underset{\mathrm{eval}}{=}(-1)^n B_n.
\end{align*}
Therefore, we see that the umbral method is equivalent to the linearity of expectation.

Furthermore, there is a conjugate symbol to $\mathcal{B}$, which is the \textit{uniform symbol} $\mathcal{U}$. This uniform symbol acts on an arbitrary function by
\[
f\left( x+\mathcal{U} \right) = \int_{0}^{1} f\left( x+u \right) du
\]
and corresponds to a continuous random variable distributed uniformly over the interval $[0,1]$. The Bernoulli and the uniform symbol essentially cancel each other, as is seen from the generating function identity
\[
\exp\left(z\B  \right)\exp\left(z\U  \right)= \frac{z}{e^z-1}\frac{e^z-1}{z}=1=\exp\left(z\left( \mathcal{U} +\mathcal{B}  \right) \right) .
\]
This is equivalent to the identity $\E\left(\U+\B \right)^n = \delta_{n,0}$ for all $n \in \mathbb{N}$, where the $\delta$ is the Kronecker symbol. Therefore, for any sufficiently smooth analytic function, 
\[
f\left( z+\mathcal{U}+\mathcal{B} \right)=f\left( z \right).
\]

\section{Poly-Bernoulli Umbra}\label{sec3}
We now wish to use the $\U$ and $\B$ symbols to characterize the poly-Bernoulli numbers as moments. We also prove an integral transform in terms of the Bernoulli-Barnes numbers. Therefore, we introduce the following nonlinear symbol below, which is the key innovation of our paper:
\begin{defn}\label{polysymdef}
Consider $k$ independent Bernoulli umbr{\ae}  $\left\{ \mathcal{B}_{i}\right\} _{1\le i\le k}$
and $\left( k-1 \right)$ independent uniform umbr{\ae} $\left\{ \mathcal{U}_{i}\right\} _{1\le i\le k-1}.$ Then define the \textit{poly-Bernoulli umbra} $\B^{(k)}$ as 
\begin{equation}
\label{symbol B}
\B^{(k)}:= 1+\sum_{l=1}^k \B_l\prod_{j=l}^{k-1}\U_j =  1+ \mathcal{B}_{k}+\mathcal{U}_{k-1}\mathcal{B}_{k-1}+\dots+\mathcal{U}_{k-1}\cdots \mathcal{U}_{1}\mathcal{B}_{1}.
\end{equation}
\end{defn}
Note that from this definition, we have the symbolic recursion 
\[
\B^{(k)} = 1+\B_k + \U_{k-1}\left( \B^{(k-1)}-1 \right),
\]
and $\B^{(1)}=\B+1$ (meaning $\B^{(1)}_n = B_n\left(1\right)$) reduces to the usual Bernoulli umbra. 

Kaneko and Arakawa also defined a companion sequence $C_{n}^{\left(k\right)}$ to the poly-Bernoulli numbers \cite{ArakawaKaneko}, with generating function

\begin{equation}\label{polydefC}
\sum_{n\ge0}\frac{C_{n}^{\left(k\right)}}{n!}t^{n}=\frac{\Li_{k}\left(1-e^{-t}\right)}{e^{t}-1},
\end{equation}
which suggests the definition of the new symbol $\mathcal{C}^{\left( k \right)}$ as
\begin{equation}
\mathcal{C}^{\left(k\right)}:=\mathcal{B}^{\left(k\right)}-1.\label{eq:BCCB}
\end{equation}
Therefore, with the notation of Definition \ref{polysymdef},
\[
\mathcal{C}^{\left(k\right)}:=\sum_{l=1}^k \B_l\prod_{j=l}^{k-1}\U_j .
\]
In terms of the sequences $\{B_{n}^{\left( k \right)}\}$ and $\{C_{n}^{\left( k \right)}\}$, this implies
\[
C_{n}^{\left(k\right)}=\sum_{l=0}^{n}\binom{n}{l}\left(-1\right)^{n-l}B_{n}^{\left(k\right)},\thinspace\thinspace B_{n}^{\left(k\right)}=\sum_{l=0}^{n}\binom{n}{l}C_{n}^{\left(k\right)}.
\]
Notice that the symbol $\mathcal{C}^{\left(k\right)}$ satisfies the recurrence
\[
\mathcal{C}^{\left(k\right)}=\mathcal{B}_k+\mathcal{U}_{k-1}\mathcal{C}^{\left(k-1\right)}
\]
with initial value $\mathcal{C}^{\left(1\right)}=\mathcal{B},$ the Bernoulli symbol.
\begin{thm}\label{thm2}
The poly-Bernoulli polynomials are  equal to
\begin{equation}
B_{n}^{\left(k\right)}\left(z\right)=\left(z+\B^{(k)}\right)^{n},\label{eq:umbral representation}
\end{equation}
which is equivalent to the multiple integral representation
\begin{equation}
B_n^{(k)}(z) = \left(\frac{\pi}{2}\right)^k \int_{\R^k} \prod_{l=1}^k \sech^2(\pi w_l)dw_l \int_{[0,1]^{k-1}} \prod_{l=1}^{k-1} du_l   \left(1+z+ \sum_{l=1}^k \left( iw_l-\frac12 \right)\prod_{j=l}^{k-1}u_j \right)^n.
\end{equation}
As a consequence, there exists a measure $\mu_{k}:\mathbb{R}\to\mathbb{R}$ such that the
poly-Bernoulli polynomial can be expressed as the moment
\[
B_{n}^{\left(k\right)}\left(z\right)=\int\left(z+w\right)^{n}d\mu_k\left(w\right).
\]
\end{thm}
\begin{proof}
Our starting point is the iterated integral representation of the generating function of the poly-Bernoulli numbers 
\cite{Kaneko}
$$\sum_{n\ge0}\frac{B_{n}^{\left(k\right)}}{n!}t^{n}= e^t\frac{1}{e^t-1} \int_0^{t} \frac{dt_k}{e^{t_k-1}}\int_{0}^{t_k}\frac{dt_{k-1}}{e^{t_{k-1}}-1} \int_0^{t_{k-1}}\cdots \int_0^{t_3} \frac{dt_2}{e^{t_2}-1}\int_0^{t_2}dt_1.\label{eq:iterated}$$
where the operator $\D=\frac{1}{e^{t}-1}\int_{0}^{t}$ appears $k-1$
times, so that
\[
e^{t\B^{\left(k\right)}}=e^{t}\D^{k-1}\frac{t}{e^{t}-1}.
\]
Using \eqref{bernint}, the action of the operator $\D$ on the initial function $\frac{t}{e^{t}-1}$
is expressed in terms of Bernoulli and uniform symbols as
\[
\D\frac{t}{e^{t}-1}=\D e^{t\B_{1}}=\frac{1}{e^{t}-1}\int_{0}^{t}e^{t\B_{1}}dt=\frac{1}{e^{t}-1}\int_{0}^{1}te^{tu\B_{1}}du=e^{t\B_{2}}e^{t\U_{1}\B_{1}}.
\]
Hence the operator $\D$ replaces $\B_{1}=\C^{(1)}$ with $\U_{1}\B_{1}+\B_{2}=\C^{(2)}$. We similarly deduce
\[
\D^{2}\frac{t}{e^{t}-1}=\D e^{t\left(\B_{2}+\U_{1}\B_{1}\right)}=e^{t\B_{3}}e^{tU_{2}\left(\B_{2}+\U_{1}\B_{1}\right)}
\]
and more generally 
\[
e^{t}\D^{k-1}\frac{t}{e^{t}-1}=e^{t\left(1+\B_{k}+\U_{k-1}\B_{k-1}+\dots+\U_{k-1}\dots \U_{1}\B_{1}\right)} = e^{t\B^{(k)}}.
\]
To pass from poly-Bernoulli numbers to poly-Bernoulli polonomials, we multiply both generating functions by $e^{zt}$, and then compare coefficients of $t^n$ to produce the desired result.
\end{proof}

Note that this procedure has converted a nested $k$--fold integral, in which the limits of the inner integrals depend on variables of integration, into a $(2k-1)$--fold integral over uniform bounds. As another consequence of the representation (\ref{eq:umbral representation}),
we have an expression for the poly-Bernoulli polynomials in terms
of the Bernoulli-Barnes polynomials. The Bernoulli-Barnes polynomials $\Bf_n^{(k)}$, a vectorized generalization of the Bernoulli polynomials, are defined by the generating function
\[
\sum_{n\ge0}\frac{\mathfrak{B}_{n}^{\left(k\right)}\left(a_{1},\dots,a_{k};z\right)}{n!}t^{n}=e^{zt}\prod_{i=1}^{k}\frac{t}{e^{a_{i}t}-1}.
\]
\begin{thm}\label{thm3}
We have the $(k-1)$--dimensional integral transform
\begin{align*}
B_{n}^{\left(k\right)}\left(z\right) & =\mathfrak{B}_{n}^{\left(k\right)}\left(\U_{k-1}\cdots \U_{1},\U_{k-1}\cdots \U_{2},\dots,\U_{k-1},1;z+1\right) \prod_{i=1}^{k-1}\U_i^i\\
 & =\int_{\left[0,1\right]^{k-1}}\mathfrak{B}_{n}^{\left(k\right)}\left(u_{k-1}\cdots u_{1},u_{k-1}\cdots u_{2},\dots,u_{k-1},1;z+1\right)u_{1}u_{2}^{2}\cdots u_{k-1}^{k-1}du_{1}\cdots du_{k-1}.
\end{align*}
\end{thm}
\begin{proof}
The main result of \cite{Lin1} was the umbral characterization
$$\Bf_n^{(k)}(a_1,\ldots,a_k;z) = \frac{1}{\prod_{i=1}^k a_i}\left(z + \sum_{i=1}^k a_i\B_i\right)^n,$$
where the $\{\B_i\}$ are independent Bernoulli umbr\ae. Comparing this with the umbral representation \eqref{eq:umbral representation}
$$B_{n}^{\left(k\right)}\left(z\right)=\left(z+\B^{(k)}\right)^{n} =\left(z+1+\sum_{l=1}^k \B_l\prod_{j=l}^{k-1}\U_j\right)^{n}  $$
and taking $a_1=1, a_i = \prod_{j=1}^{i-1}\U_j, i \geq 2,$ gives the first result. The second follows by explicitly specifying the action of the uniform operators $\U_i$.
\end{proof}

The following result shows how a change of variables in the integral representation in Theorem \ref{thm3} allows us to replace the cumulative products of variables $\U_{i}$ by linear terms, at the price of a more complicated integration domain.

\begin{thm}
The Arakawa-Kaneko polynomials are expressed as
\[
B_{n}^{\left(k\right)}\left(z\right)=
\mathfrak{B}_{n}^{\left( k \right)}
\left( \V_{1},\dots,\V_{k-1};z+1 \right)\prod_{i=1}^{k-1}\V_{i},
\]
where 
$\left\{ \V_{i}\right\} _{1\le i\le k}$ are random variables
with the mutual probability density
\begin{align*}
f_{V_{1},\dots,V_{k-1}}\left(v_{1},\dots v_{k-1}\right) & =\begin{cases}
\frac{1}{v_{2}\dots v_{k-1}} & 0\le v_{1}\le\dots\le v_{k-1}\le1,\\
0 & \text{else.}
\end{cases}
\end{align*}
This yields the integral representation
\[
B_{n}^{\left(k\right)}\left(z\right)=\int_{0\le v_{1}\le\dots\le v_{k-1}\le1} v_{1}\mathfrak{B}_{n}^{\left( k \right)}
\left( v_{1},\dots,v_{k};z+1 \right) dv_{1}\dots dv_{k-1}
\]
and, in terms of Bernoulli umbr{\ae},
\[
B_{n}^{\left(k\right)}\left(z\right)=\int_{0\le v_{1}\le\dots\le v_{k-1}\le1} 
\left( 1+z+ \B_{k} +v_{k-1} \B_{k-1} +\dots +v_{1} \B_{1} \right)^{n}
\frac{dv_{1}\dots dv_{k-1}}{v_{2}\dots v_{k-1}}
\]
\end{thm}

\begin{proof}
Perform the change of variables 
\[
\begin{cases}
v_{k-1} & =u_{k-1}\\
v_{k-2} & =u_{k-1}u_{k-2}\\
\vdots\\
v_{1} & =u_{k-1}\dots u_{1}
\end{cases}
\]
in the integral representation in Theorem \ref{thm3}; the Jacobian is then
\[
J=\frac{\partial v}{\partial u}=\left[\begin{array}{ccccc}
u_{k-1}\dots u_{2} & 0 & \cdots & \cdots & 0\\
* & u_{k-1}\dots u_{3} & 0 & \cdots & 0\\
 &  &  &  & \vdots\\
 &  &  & u_{k-1}   & 0\\
 * &  &  & * & 1
\end{array}\right],
\]
with determinant
\[
\vert J\vert  = \prod_{i=2}^{k}u_{i}^{i-1}=\prod_{i=2}^{k-1}v_{i}.
\]
The product of uniform densities
\[
\mathbbm{1}_{\left[0,1\right]}\left(u_{1}\right)\times\dots\times\mathbbm{1}_{\left[0,1\right]}\left(u_{k-1}\right)
\]
is transformed into
\[
\mathbbm{1}_{\left[0,1\right]}\left(\frac{v_{1}}{v_{2}}\right)\mathbbm{1}_{\left[0,1\right]}\left(\frac{v_{2}}{v_{3}}\right)\times\dots\times
\mathbbm{1}_{\left[0,1\right]}\left(\frac{v_{k-2}}{v_{k-1}}\right)
\mathbbm{1}_{\left[0,1\right]}\left(v_{k-1}\right)
=\begin{cases}
1, & 0\le v_{1}\le\dots\le v_{k-1}\le1\\
0, & \text{else}
\end{cases},
\]
which is the desired result.
\end{proof}

\section{Consequences}\label{sec4}
This symbolic representation enables us to quickly prove various recursions for the poly-Bernoulli polynomials. All the identities in this section have analogs for the companion $C^{(n)}$ sequence with identical proofs, which we omit.
\begin{thm}
The poly-Bernoulli polynomials satisfy the recurrences
\[
B_{n}^{\left(k\right)}\left(z\right)=\sum_{m=0}^{n}\binom{n}{m}B_{n-m}^{\left(k-1\right)}\sum_{l=0}^{m}\binom{m}{l}\frac{\left(-1\right)^{m-l}B_{l}\left(z+1\right)}{n-l+1}.
\]
and
\[
C_{n}^{\left(k\right)}\left(z\right)=\sum_{m=0}^{n}\binom{n}{m}B_{m}\left( z \right)\frac{C^{\left( k-1 \right)}_{n-m}}{n-m+1}.
\]
\end{thm}
\begin{proof}
Begin with the umbral representation (\ref{eq:umbral representation}). We then recursively write
\begin{align*}
B_{n}^{\left(k\right)}\left(z\right) & =\left(z+1+\mathcal{U}_{k-1}\mathcal{B}^{\left(k-1\right)}+\mathcal{B}_{k}-\mathcal{U}_{k-1}\right)^{n}\\
&= \sum_{m=0}^n \binom{n}{m}\left[\U_{k-1}\B^{(k-1)}\right]^{n-m}\left(\B_k+z+1 -\U_{k-1}\right)^m\\
 & =\sum_{m=0}^{n}\binom{n}{m}\left(\B^{(k-1)}\right)^{n-m}\sum_{l=0}^{m}\binom{m}{l}\left(\B_k+z+1\right)^l\left(-1\right)^{m-l}\U_{k-1}^{m-l}\U_{k-1}^{n-m}\\
 &\underset{\mathrm{(eval)}}{=}\sum_{m=0}^{n}\binom{n}{m}B_{n-m}^{\left(k-1\right)}\sum_{l=0}^{m}\binom{m}{l}\frac{\left(-1\right)^{m-l}B_{l}\left(z+1\right)}{n-l+1}.
\end{align*}
Note that the nonlinear $\B^{(k-1)}$ umbra is only dependent on $\{\B_1,\ldots,\B_{k-1}\}$ and $\{\U_1,\ldots,\U_{k-2}\}$, which means that we can apply the ``eval'' functional to it independently.
\end{proof}
\begin{rem}
This result appears as \cite[Thm 1.2]{hamahata} under the form
\[
B_{n}^{\left(k\right)}\left(z\right)=\sum_{m=0}^{n}\binom{n}{m}\left(-1\right)^{m}B_{n-m}^{\left(k-1\right)}\sum_{l=0}^{m}\binom{m}{l}\frac{\left(-1\right)^{l}}{n-l+1}B_{l}\left(z\right),
\]
which differs by the term $B_{l}\left(z\right)$ from our result and has been numerically verified as incorrect. 
\end{rem}

The next several results feature a nonlinear product of symbols of the form $\U \B^{(k-1)}$, which, to our knowledge, is the first time such a nonlinear function of symbols has appeared in the literature. It also features the negatively indexed symbol $\left(\B^{(k)}\right)^{-1}$, which is evaluated as the negative moment
$$
\left(z+\B^{(k)}\right)^{-1} = \left(\frac{\pi}{2}\right)^k \int_{\R^k} \prod_{l=1}^k \sech^2(\pi w_l)dw_l \int_{[0,1]^{k-1}} \prod_{l=1}^{k-1} du_l   \frac{1}{1+z+ \sum_{l=1}^k \left( iw_l-\frac12 \right)\prod_{j=l}^{k-1}u_j }.
$$
The evaluation of these negative moments will be explicitly described in the next section.

\begin{thm}
\label{thm7}
The poly-Bernoulli polynomials satisfy the connection relation
\[
B_{n}^{\left(k\right)}\left(z\right)=\frac{1}{n+1} \sum_{l=0}^n \binom{n+1}{l+1}B_{n-l}(z+1) B_l^{(k-1)}(-1)
\]
and
\[
C_{n}^{\left( k \right)}\left( z \right) = \frac{1}{n+1}\sum_{l=0}^{n}\binom{n+1}{l+1}B_{n-l}\left( z \right)C_{l}^{\left( k-1 \right)}.
\]
\end{thm}
\begin{proof}
Begin with the umbral representation (\ref{eq:umbral representation}),
$$B_n^{(k)}(z) = \left(1+z+\B_k +\U_{k-1}\left(\B^{(k-1)}-1 \right)\right)^n, $$
and directly apply the $\U_{k-1}$ integration operator to obtain
\begin{align*}
B_n^{(k)}(z) &= \frac{\left(\B^{(k-1)}-1\right)^{-1}}{n+1}\left[\left(1+z+\B_k +\B^{(k-1)}-1\right)^{n+1} - \left(1+z+\B_k\right)^{n+1}\right] \\
&= \frac{1}{n+1} \sum_{l=1}^{n+1} \binom{n+1}{l} (1+z+\B_k)^{n+1-l} \left(\B^{(k-1)} -1 \right)^{l-1} \\
&= \frac{1}{n+1} \sum_{l=0}^n \binom{n+1}{l+1}B_{n-l}(z+1) B_l^{(k-1)}(-1).
\end{align*}
\end{proof}

\begin{rem}
The Fourier coefficients of the periodic poly-Bernoulli polynomials $C_{n}^{\left( k \right)} \left( \left\{ z\right\}  \right)$, defined by  
\[
c_{l}^{\left( k,n \right)} = \int_{0}^{1} C_{n}^{\left( k \right)} \left( \left\{ z\right\}  \right) e^{-\imath 2 \pi lz} dz,
\]
are computed in \cite{Kim} as
\[
c_{l}^{\left( k,n \right)} = 
\begin{cases}
-\frac{n!}{\left(\imath2\pi l\right)^{n+1}}\sum_{p=1}^{n}\frac{C_{p-1}^{\left(k-1\right)}}{p!}\left(\imath2\pi l\right)^{p} & l \ne 0\\
\frac{1}{n+1}C_{n}^{\left(k-1\right)} & l=0.
\end{cases}
\]
Thus the Fourier expansion of the periodic poly-Bernoulli polynomials $C_{n}^{\left( k \right)} \left( \left\{ z\right\}  \right)$ reads
\[
C_{n}^{\left(k\right)}\left(z\right)=\frac{1}{n+1}\sum_{p=1}^{n+1}\binom{n+1}{p}C_{p-1}^{\left(k-1\right)}B_{n-p+1}\left(\left\{ z\right\} \right).
\]
As a consequence, both identities in Theorem \ref{thm7} can be interpreted as the Fourier expansion of the poly-Bernoulli polynomials $B_{n}^{\left( k \right)} \left( \left\{ z\right\}  \right)$ and $C_{n}^{\left( k \right)} \left( \left\{ z\right\}  \right)$.
\end{rem}

The next result is a higher-order analog of the identity on Bernoulli symbols $f(1+\B) = f(\B)+f'(0)$, for smooth analytic functions $f$.
\begin{thm}
The poly-Bernoulli polynomials satisfy the difference identities
\[
B_{n}^{\left(k\right)}\left(z\right)-B_{n}^{\left(k\right)}\left(z-1\right)= \sum_{l=1}^n \binom{n}{l} z^{n-l}B^{(k-1)}_{l-1}(-1).
\]
and
\[
C_{n}^{\left(k\right)}\left(z+1\right)-C_{n}^{\left(k\right)}\left(z\right)= \sum_{l=1}^{n} \binom{n}{l} z^{n-l}C^{(k-1)}_{l-1}.
\]
\end{thm}
\begin{proof}
We begin by applying the classical relation $f(1+\B) = f(\B)+f'(0)$ \cite{Gessel} to the outermost $\B_k$ symbol, then utilize the trick from the previous theorem and directly apply the $\U_{k-1}$ integration operator and invoke symbols at negative indices:
\begin{align*}
B_n^{(k)}(z) &= \left(1+z+\B_k +\U_{k-1} \left( \B^{(k-1)}-1\right)\right)^n \\
&= \left(z+\B_k +\U_{k-1}\left( \B^{(k-1)}-1\right)\right)^n +n\left(z +\U_{k-1}\left( \B^{(k-1)}-1\right)\right)^{n-1} \\
&= \left(z-1+\B^{(k)}\right)^n + \left(\B^{(k-1)}-1\right)^{-1} \left[ \left(z-1+\B^{(k-1)}\right)^n - z^n \right] \\
&= \left(z-1+\B^{(k)}\right)^n + \sum_{l=1}^n \binom{n}{l} z^{n-l}\left(\B^{(k-1)}-1\right)^{l-1}.
\end{align*}
Applying the eval functional completes the proof.
\end{proof}
This identity can be extended to all smooth analytic functions by linearity, since we've verified it for monomials. For $k \geq 2$, this gives the symbolic extension
$$f(1+z+\B^{(k)})-f(z+\B^{(k)}) = \frac{f(z+\B^{(k-1)})-f(z)}{\B^{(k-1)}}, $$
which converts a forwards difference in $z$ into a discrete derivative with respect to $\B^{(k-1)}$. Note that for $k=1$ we instead have $f(1+z+\B)-f(z+\B) = f'(z)$, so that as $k \to 1$ we're performing a discrete approximation to the continuous derivative $f'(z)$, with respect to $\B^{(k-1)}$. This suggests that other functional identities involving the Bernoulli symbol and derivatives ought to have generalizations to the poly-Bernoulli numbers.

Analogously the Bernoulli symbol identity \cite{Gessel}
\[
f(m+\B) = f(\B) + \sum_{i=0}^{m-1}f'(i)
\] 
extends to the poly-Bernoulli case as follows.
\begin{thm}
The poly-Bernoulli polynomials satisfy the higher-order difference identity
\[
B_{n}^{\left(k\right)}\left(m+z\right)-B_{n}^{\left(k\right)}\left(z\right)=\sum_{l=1}^n \binom{n}{l} B^{(k-1)}_{l-1}(-1) \left( \frac{B_{n-l+1}(m+z+1)-B_{n-l+1}(z+1)}{n-l+1}\right)
\]
and
\[
C_{n}^{\left(k\right)}\left(m+z\right)-C_{n}^{\left(k\right)}\left(z\right)=\sum_{l=1}^n \binom{n}{l} C^{(k-1)}_{l-1} \left( \frac{B_{n-l+1}(m+z)-B_{n-l+1}(z)}{n-l+1}\right).
\]
\end{thm}
\begin{proof}
We begin by applying the classical relation $f(m+\B) = f(\B) + \sum_{i=0}^{m-1}f'(i)$:
\begin{align*}
B_n^{(k)}(m+z) &= \left(m+1+z+\B_k +\U_{k-1}\left( \B^{(k-1)}-1\right)\right)^n \\
&= \left(1+z+\B_k +\U_{k-1}\left( \B^{(k-1)}-1\right)\right)^n  +n\sum_{i=0}^{m-1}\left(i+1+z +\U_{k-1}\left( \B^{(k-1)}-1\right)\right)^{n-1} \\
&= \left(z+\B^{(k)}\right)^n  +\left(\B^{(k-1)}-1\right)^{-1}\sum_{i=0}^{m-1}\left[\left(i+1+z +\B^{(k-1)}-1\right)^{n} - \left(i+1+z \right)^{n} \right] \\
&= \left(z+\B^{(k)}\right)^n  + \sum_{i=0}^{m-1}\sum_{l=1}^n \binom{n}{l}(1+z+i)^{n-l}\left(\B^{(k-1)}-1\right)^{l-1}.
\end{align*}
We now apply the generalized Faulhaber formula
$$ \sum_{i=0}^{m-1} (1+z+i)^{n-l} = \frac{B_{n-l+1}(m+z+1)-B_{n-l+1}(z+1)}{n-l+1} $$
to eliminate the outer summation, which completes the proof.

\end{proof}

\section{a new approach to MZVs}\label{sec5}
A consequence of the symbolic representation \eqref{symbol B} is that it allows us to represent the analytic continuation of the Arakawa zeta function as a Bernoulli symbol to a negative power. The easiest case of this result is the classical representation 
\[
\zeta\left( n \right) = \left( -1 \right)^n
\frac{B_{1-n}}{1-n},\,\, n\ge 2,
\]
which should be understood as a consequence of the integral representation
\eqref{bernint}
\[
\zeta\left( n \right) = 
\frac{\left( -1 \right)^n}{1-n}
\frac{\pi}{2}\int_{-\infty}^\infty \left( \imath t-\frac{1}{2} 
\right)^{ 1-n} \sech^2\left(  \pi t \right) dt.
\]

Remark that Theorem \ref{negmoments} essentially computes the negative moments $(\B^{(k)})^{-m}$ . For $k=1$, these negative moments can be separately evaluated; we begin with the nonlinear umbral evaluation \cite[Thm 2.5]{Dixit}
\begin{equation}
\label{1/B}
\log (\B+z) = \psi\left( \frac12+ \bigg| z-\frac12 \bigg|\right), \end{equation}
where $\psi(z) := \frac{\Gamma'(z)}{\Gamma(z)}$ is the digamma function. Then, by noting $(\B+z)^{-1} = \frac{d}{dz} \log(\B+z) $, and taking iterated derivatives, we arrive at the evaluation
\begin{equation}
\label{psi}
(\B+z)^{-m}  = 
\begin{cases}
\frac{(-1)^{m-1}}{(m-1)!} \psi^{(m)}(z), &z>\frac12 \\
-\frac{1}{(m-1)!} \psi^{(m)}(1-z), &z<\frac12 
\end{cases},
\end{equation}
which is nondifferentiable at $z=\frac12$. Here, $\psi^{(m)}$ is the $m$--th order polygamma function. At  $z=1$, this specializes to $(\B+1)^{-m} = m\zeta(m+1)$.

Before proving a new symbolic representation for the Arakawa zeta function, we introduce some lemmas about the \textit{Barnes zeta function}, a generalization of the Riemann zeta function parametrized by a vector of $k$ real variables $(a_1,a_2,\ldots,a_k)$ as
$$\zeta_k(s,z| a_1,\ldots,a_k) :=  \sum_{m_1,\ldots,m_k \geq 0} \frac{1}{(z+a_1m_1+\cdots + a_km_k)^s}.$$
Much of the material from this section is derived from the important manuscript \cite{Ruijsenaars} of Ruijsenaars, which we highly recommend. The following integral representations are crucial to our argument.
\begin{lem}\label{lem1}\cite[Eqn 1.6]{Ruijsenaars}
With values of the parameters such that both sides converge, we have the iterated integral 
\begin{equation}\label{eq2}
\zeta_k\left(s,w+\frac12\sum_{i=1}^k a_i\right) = \left(\frac{\pi}{2}\right)^k \frac{\Gamma(s-k)}{\Gamma(s)}\int_{{\R}^k} \prod_{n=1}^k \frac{1}{a_n^2}\sech^2\left(\frac{\pi u_n}{a_n}\right)  \frac{1}{\left(w - i\sum_{n=1}^ku_n\right)^{s-k}}dx_1\cdots dx_k. 
\end{equation}
\end{lem}
\begin{lem}\label{lem2}\cite[Eqn 3.2]{Ruijsenaars}
With values of the parameters such that both sides converge, we have the Barnes-Mellin transform
\begin{equation}
\zeta_k(s,w)=\frac{1}{\Gamma(s)}\int_0^\infty  \prod_{j=1}^k \frac{1}{1-e^{-a_jt}}t^{s-1} e^{-wt}dt.
\end{equation}
\end{lem}

The Barnes zeta function satisfies a typical zeta function--Bernoulli number duality; given the Bernoulli--Barnes numbers defined before Theorem \ref{thm3}, we have the identity on analytic continuations
$$\zeta_k(-N,z|a_1,\ldots,a_k)  = (-1)^k \frac{N!}{(n+N)!} \mathfrak{B}_{n+N}^{\left(k\right)}\left(a_{1},\dots,a_{k};z\right).$$
Therefore, we heuristically expect Theorem \ref{thm3} to have an analog for the Arakawa zeta, by mapping $s \mapsto -N$. This is how we proceed; we write the Arakawa zeta function as an integral transform of the Barnes zeta function, and then apply some known results for the Barnes zeta.

\begin{thm}\label{negmoments}
A symbolic representation for the Arakawa zeta function is
\begin{equation}
\label{symbolic representation}
 \zeta_k\left( m \right) = \left(-1\right)^{m}\C^{\left( k \right)}_{-m}.
\end{equation}
\end{thm}
\begin{proof}
Heuristically, the idea behind this identity is identifying the Arakawa zeta as a Barnes-Mellin transform 
\begin{align*}
\zeta_{k}\left(m,w\right) & =\frac{1}{\Gamma\left(m\right)}\int_{0}^{+\infty}t^{m-1}e^{-wt}\frac{\Li_{k}\left(1-e^{-t}\right)}{e^{t}-1}dt\\
 & =\frac{1}{\Gamma\left(m\right)}\int_{0}^{+\infty}t^{m-1}e^{-wt}e^{t\C^{\left(k\right)}}dt\\
 & =\frac{1}{\Gamma\left(m\right)}\int_{0}^{+\infty}t^{m-1}e^{-t\left(w-\C^{\left(k\right)}\right)}dt\\
 & =\left(w-\C^{\left(k\right)}\right)^{-m}.
\end{align*}
However, we need to make this approach rigorous. 
Denote by
\[
f_{L}\left(x\right)=\frac{\pi}{2}\sech^{2}\left(\pi x\right)
\]
the square hyperbolic secant distribution. Then, by combining Lemmas \ref{lem1} and \ref{lem2}, with the change of variables $u_n \mapsto a_n x_n$ on the left-hand side, we obtain

\[
\left(\prod_{i=1}^{k}\frac{1}{\left(m-i\right)a_{i}}\right)\int_{\mathbb{R}^{k}}\frac{\prod_{i=1}^{k}f_{L}\left(x_{i}\right)}{\left(w+\sum_{i=1}^{k}a_{i}\left(\imath x_{i}-\frac{1}{2}\right)\right)^{m-k}}dx_{1}\dots dx_{k}=\frac{1}{\Gamma\left(m\right)}\int_{0}^{+\infty}t^{m-1}e^{-tw}\prod_{i=1}^{k}\frac{1}{1-e^{-a_{i}t}}dt.
\]
Replacing $m-k$ by $m$ and simplifying produces
\[
\int_{\mathbb{R}^{k}}\frac{\prod_{i=1}^{k}f_{L}\left(x_{i}\right)}{\left(w+\sum_{i=1}^{k}a_{i}\left(\imath x_{i}-\frac{1}{2}\right)\right)^{m}}\prod_{i=1}^{k}dx_i=\frac{1}{\Gamma\left(m\right)}\int_{0}^{+\infty}t^{m-1}e^{-wt}\left(\prod_{i=1}^{k}\frac{-a_{i}t}{e^{-a_{i}t}-1}\right)dt.
\]
Choosing $a_{k}=-1$ and $a_{k-i}=-u_{k-1}\dots u_{k-i},\thinspace\thinspace1\le i\le k-1$, now gives
\begin{align*}
\int_{\mathbb{R}^{k}}   &   \frac{\prod_{i=1}^{k}f_{L}\left(x_{i}\right)}{\left(w-\left(\imath x_{k}-\frac{1}{2}\right)-\sum_{i=1}^{k-1}u_{k-1}\dots u_{i}\left(\imath x_{i}-\frac{1}{2}\right)\right)^{m}}\prod_{i=1}^{k}dx_i   \\
&=\frac{1}{\Gamma\left(m\right)}\int_{0}^{+\infty}t^{m-1}e^{-wt}\left(\frac{t}{e^t-1}\prod_{i=1}^{k-1}\frac{u_{k-1}\dots u_{i}t}{e^{u_{k-1}\dots u_{i}t}-1}\right)dt.
\end{align*}
Integrating each $u_{i},\thinspace\thinspace1\le i\le k-1$ over the
interval $\left[0,1\right]$ gives
\begin{align}\label{keystep}
\int_{\left[0,1\right]^{k-1}}\int_{\mathbb{R}^{k}}&\frac{\prod_{i=1}^{k}f_{L}\left(x_{i}\right)}{\left(w-\left(\imath x_{k}-\frac{1}{2}\right)-\sum_{i=1}^{k-1}u_{k-1}\dots u_{i}\left(\imath x_{i}-\frac{1}{2}\right)\right)^{m}}\prod_{i=1}^{k}dx_i\prod_{i=1}^{k-1}du_{i}\nonumber \\
& =\frac{1}{\Gamma\left(m\right)}\int_{0}^{+\infty}t^{m-1}e^{-wt}\left(\int_{\left[0,1\right]^{k-1}}\frac{t}{e^{t}-1}\prod_{i=1}^{k-1}\frac{u_{k-1}\dots u_{i}t}{e^{u_{k-1}\dots u_{i}t}-1}du_i\right)dt.
\end{align}
This step essentially writes the Arakawa zeta function as a $(k-1)$ dimensional integral transform of the Barnes zeta function, a result parallel to Theorem \ref{thm3}.  From the definition of the poly-Bernoulli umbra, the left-hand side of Equation \eqref{keystep} is written as 
\[
\frac{1}{\left(w-\C^{\left( k \right)}\right)^{m}}.
\]
Meanwhile, we now recognize the inner integral on the right-hand side of Equation \eqref{keystep} as the generating
function of the $C^{\left(k\right)}$ sequence:
\[
\sum_{n\ge0}\frac{C_{n}^{\left(k\right)}}{n!}t^{n}=\frac{\Li_{k}\left(1-e^{-t}\right)}{e^{t}-1}.
\]
To see this, begin with Theorem \ref{thm2} 
$$\frac{Li_k(1-e^{-t})}{e^{t}-1} = \sum_{n=0}^\infty \frac{C_n^{(k)}}{n!}t^n  = \exp \left(t\left( \B_k+\U_{k-1}\B_{k-1}+\U_{k-1}\cdots\U_1 \B_1 \right)\right).$$
Then expand each $\B$ symbol through its generating function as
$$ \frac{t}{e^t-1}\times \frac{t\U_{k-1}}{e^{t\U_{k-1}}-1}\times \cdots \times \frac{t\U_{k-1}\cdots \U_1}{e^{t\U_{k-1}\cdots \U_1}-1}. $$
Now expand each $\U$ operator as an integral, giving the key identity
\begin{equation*}\label{keylem}
\frac{\Li_k(1-e^{-t})}{1-e^{-t}}  = \int_{[0,1]^{k-1}} t^k \prod_{i=1}^{k-1}u_i^i \prod_{i=1}^{k} \frac{1}{e^{t\sum_{j=i}^{k-1} u_i}-1}du_1\cdots du_{k-1}.
\end{equation*}
Therefore, the right-hand side is equal to the Arakawa-Kaneko zeta 
\[
\frac{1}{\Gamma\left(m\right)}\int_{0}^{+\infty}t^{m-1}e^{-wt}\frac{\Li_{k}\left(1-e^{-t}\right)}{e^{t}-1}dt=\zeta_{m}\left(k,w\right).
\]
Evaluating at $w=0$ yields the result.
\end{proof}
As an aside, this proof highlights the fasincating (and highly nontrivial) interplay between Bernoulli numbers and zeta functions. This proof was ultimately based on Theorem \ref{thm3}, a new representation for poly-Bernoulli numbers. This allowed us to prove an integral transform for Arakawa zeta function, which when analytically continued re-specializes to results on the poly-Bernoulli numbers. This result provides new integral representations for the starred MZVs, some of which we record here.
\begin{itemize}
\item
For $k=2,m=1,$ we have
\begin{equation}
\label{-z(3)}
\frac{1}{\B_{2}+\U_{1}\B_{1}}=-1.20206=-\zeta^{*}\left(3\right)=-\zeta\left(3\right)=-\zeta_{2}\left(1\right)
\end{equation}
so that
\[
\zeta=\left(3\right)=-\frac{\pi^{2}}{4}\int_{\mathbb{R}^{2}}\int_{0}^{1}\frac{{\sech}^{2}\left(\pi w_1\right){\sech}^{2}\left(\pi w_2\right)}{\imath w_2-\frac{1}{2}+u\left(\imath w_1-\frac{1}{2}\right)}dudw_2dw_1.
\]
\item
For $k=2,m=2,$ we have
\[
\frac{1}{\left(\B_{2}+\U_{1}\B_{1}\right)^{2}}=1.3529=\zeta^{*}\left(3,1\right)=\zeta_{2}\left(2\right)=\zeta\left(3,1\right)+\zeta\left(4\right)
\]
so that
\[
\zeta^{*}\left(3,1\right)=\frac{\pi^{2}}{4}\int_{\mathbb{R}^{2}}\int_{0}^{1}\frac{{\sech}^{2}\left(\pi w_1\right){\sech}^{2}\left(\pi w_2\right)}{\left(\imath w_2-\frac{1}{2}+u\left(\imath w_1-\frac{1}{2}\right)\right)^{2}}dudw_2dw_1.
\]
This integral can also be expressed as the double integral
\[
\zeta^{*}\left(3,1\right)=\frac{\pi^{2}}{4}\int_{\mathbb{R}^{2}}{\sech}^{2}\left(\pi w_1\right){\sech}^{2}\left(\pi w_2\right)
\frac{\imath w_1-\frac{1}{2}}{\left(\imath w_2-\frac{1}{2}\right)\left(\imath \left(w_2+ w_1\right)-1\right)}dw_2dw_1.
\]
This motivates the study of whether we can simplify some of the $(k-1)$ fold integrals corresponding to the $\U_i$ operators.
\item
For $k=2,m=3,$ we have
\[
\frac{1}{\left(\B_{2}+\U_{1}\B_{1}\right)^{3}}=-1.45884=-\zeta^{*}\left(3,1,1\right)=-\zeta_{2}\left(3\right)
\]
so that
\[
\zeta^{*}\left(3,1,1\right)=-\frac{\pi^{2}}{4}\int_{\mathbb{R}^{2}}\int_{0}^{1}\frac{\text{sech}^{2}\left(\pi w_1\right)\text{sech}^{2}\left(\pi w_2\right)}{\left(\imath w_2-\frac{1}{2}+u\left(\imath w_1-\frac{1}{2}\right)\right)^{3}}dudw_2dw_1.
\]
\end{itemize}


\section{Bernoulli symbols and summation}

In this last part, we illustrate the summation mechanism performed by the Bernoulli
and uniform umbr{\ae}.
\subsection{Telescoping}
Starting with the identity
\begin{equation}
\label{1/B}
\frac{1}{x+\B}=\sum_{k\ge0}\frac{1}{\left(x+k\right)^{2}},
\end{equation}
which holds for $x>\frac{1}{2}$ as a consequence of \eqref{psi}, and when evaluated at $x=1,$ gives
\[
\zeta\left(2\right)=\frac{1}{\B+1}=-\frac{1}{\B}.
\]
This result can be expressed in integral form as
\[
\frac{\pi}{2}\int_{\mathbb{R}}\frac{{\sech}^{2}\left(\pi w\right)}{\imath w+\frac{1}{2}}dw = \zeta\left( 2 \right).
\]
Renaming $\B$ as $\B_{2}$ and replacing $x$ by $x+\U_{1}\B_{1}$ produces
\[
\frac{1}{x+\B_{2}+\U_{1}\B_{1}}=\sum_{k\ge0}\frac{1}{\left(x+k+\U_{1}\B_{1}\right)^{2}}.
\]
Expressing the $U_{1}$ symbol as an integral over the interval $\left[ 0,1 \right]$ yields
\[
\frac{1}{x+\B_{2}+\U_{1}\B_{1}}=\sum_{k\ge0}\frac{1}{\B_{1}}\left[\frac{1}{\left(x+k\right)}-\frac{1}{\left(x+k+\B_{1}\right)}\right]=\sum_{k\ge0}\frac{1}{\left(x+k\right)\left(x+k+\B_{1}\right)},
\]
and then applying identity \eqref{1/B} again gives
\[
\frac{1}{x+\B_{2}+\U_{1}\B_{1}}=\sum_{k\ge0}\frac{1}{\left(x+k\right)}\sum_{l\ge0}\frac{1}{\left(x+k+l\right)^{2}}.
\]
Evaluating at $x=1$ produces
\[
\frac{1}{1+\B_{2}+\U_{1}\B_{1}}=\sum_{k\ge1,l\ge0}\frac{1}{k}\frac{1}{\left(k+l\right)^{2}}=\zeta^{*}\left(2,1\right)=\zeta\left(2,1\right)+\zeta\left(3\right)=2\zeta\left(3\right).
\]
This result should be compared with the expression \eqref{-z(3)} above, which we rewrite as
\[
\zeta\left( 3 \right) = \frac{1}{1+\B_{2}-\U_{1}\B_{1}}.
\]
\subsection{Computing MZVs}
Instead of multiple zeta star values, multiple zeta values can be obtained  in the computation above, by noticing that each Bernoulli symbol induces a sum that starts at 0. Replacing each Bernoulli symbol $\B_{i}$ by $1+\B_{i}$ then induces sums that start at 1, producing
\[
\frac{1}{1+\B_{k}+\U_{k-1}\left( 1+\B_{k-1} \right) +\dots+\U_{k-1}\dots\U_{1}\left( 1+ \B_{1}\right)} = \zeta\left( 2, \left\{ 1\right\}^{k-1} \right).
\]
Since for each $i$, $1+\B_{i}=-\B_{i}$, this fraction is identified as
\[
-\frac{1}{\C^{\left( k \right)}} = \zeta_{k}\left( 1 \right) = \zeta^{*}\left( k+1 \right) = \zeta\left( k+1 \right)
\]
and we deduce
\[
\zeta\left( k+1 \right) = \zeta\left( 2, \left\{ 1\right\}^{k-1} \right),
\]
an identity that is a consequence of the duality property of MZVs. Further studying the effect of the $\B+1 = -\B$ periodicity on relations for MZVs and MZSVs is an intriguing possibility for further study.
\section{Acknowledgements}
This one goes out to the gorgeous views of Luxembourg City, which inspired much of Section \ref{sec4}.



\begin{thebibliography}{1}

\bibitem{ArakawaKaneko}
T. Arakawa and M. Kaneko, {Multiple zeta values, poly-Bernoulli numbers, and related zeta functions}, Nagoya Math. J., 153, 189-209 (1999)


\bibitem{Arakawa}T. Arakawa, T. Ibukiyama and K. Masanobu, {Bernoulli numbers and zeta functions}, Springer Monographs in Mathematics: With an appendix by Don Zagier, {Springer, Tokyo}, xii + 274 (2014)


\bibitem{hamahata}A. Bayad and Y. Hamahata, Polylogarithms and poly-Bernoulli
polynomials, Kyushu J. Math., 65, 15-24  (2011)

\bibitem{broadhurst}
D.J.Broadhurst and D.Kreimer,
Association of multiple zeta values with positive knots via Feynman diagrams up to 9 loops,
Physics Letters B, 393 3--4, 403-412 (1997)

\bibitem{Dixit}A. Dixit, V.-H. Moll, and C. Vignat, {The {Z}agier modification of {B}ernoulli numbers and a
              polynomial extension. {P}art {I}}, Ramanujan J., 33:3,379-422 (2014)

\bibitem{Gessel}I. M. Gessel, Applications of the classical umbral calculus, \newblock{Algebra Universalis}, 49, 397-434 (2003)

\bibitem{Lin1}L. Jiu, V. H. Moll, and C. Vignat, {A symbolic approach to some identities for
              {B}ernoulli-{B}arnes polynomials}, {Int. J. Number Theory}, 12:3, 649-662 (2016)

\bibitem{Kaneko}M. Kaneko, Poly-Bernoulli numbers Journal de Th\'{e}orie
des Nombres de Bordeaux, Tome 9: No 1, 221- 228 (1997)

\bibitem{Kim}T. Kim, D.S. Kim, S.-H. Rim and D.V. Dolgy,
Fourier series of higher-order Bernoulli functions and their applications,
Journal of Inequalities and Applications (2017) 2017:8

\bibitem{Ruijsenaars}S. N. M. Ruijsenaars, On Barnes Multiple Zeta and Gamma Functions, Advances in Mathematics, 156, 107--132 (2000)

\bibitem{Vignat1}H.-M. Srivastava and C. Vignat, Probabilistic proofs of some relationships between the
              {B}ernoulli and {E}uler polynomials, {Eur. J. Pure Appl. Math.}, 5:2, 97-107 (2012)

\bibitem{Zagier}
{D. Zagier},
\newblock {{Values of zeta functions and their applications},}
\newblock{First {E}uropean {C}ongress of {M}athematics, {V}ol.\ {II}
              {P}aris, 1992,}
\newblock {Vol. 120, 497--512} (1994)

\end{thebibliography}
\end{document}